\documentclass[reqno,twoside,11pt]{amsart}
\usepackage{cite}
\usepackage{amsmath}
\usepackage{amsfonts}
\usepackage{amssymb}
\usepackage{epsfig}
\usepackage{verbatim}

\IfFileExists{epsf.def}{\input epsf.def}{\usepackage{epsf}}

\IfFileExists{mathptmx.sty}{\usepackage{mathptmx}}{\usepackage{mathpazo}}
\usepackage{mathrsfs}

\DeclareFontFamily{OT1}{eusb}{} \DeclareFontShape{OT1}{eusb}{m}{n} {<5> <6> <7> <8> <9> <10> <11> <12> <14.4> eusb10}{}
\DeclareMathAlphabet{\eusb}{OT1}{eusb}{m}{n}

\DeclareFontFamily{OT1}{eusm}{} \DeclareFontShape{OT1}{eusm}{m}{n} {<5> <6> <7> <8> <9> <10> <11> <12> <14.4> eusm10}{}
\DeclareMathAlphabet{\eusm}{OT1}{eusm}{m}{n}

\DeclareFontFamily{OT1}{eufm}{} \DeclareFontShape{OT1}{eufm}{m}{n} {<5> <6> <7> <8> <9> <10> <11> <12> <14.4> eufm10}{}
\DeclareMathAlphabet{\mathfrak}{OT1}{eufm}{m}{n}

\DeclareFontFamily{OT1}{fraktura}{}
\DeclareFontShape{OT1}{fraktura}{m}{n} {<5> <6> <7> <8> <9> <10> <11> <12> <13> <14.4> [1.1] eufm10}{}
\DeclareMathAlphabet{\fraktura}{OT1}{fraktura}{m}{n}

\DeclareFontFamily{OT1}{cmfi}{} \DeclareFontShape{OT1}{cmfi}{m}{n} {<5> <6> <7> <8> <9> <10> <11> <12> <13> <14.4> [0.9] cmfi10}{}
\DeclareMathAlphabet{\cmfi}{OT1}{cmfi}{b}{n}

\DeclareFontFamily{OT1}{cmss}{} \DeclareFontShape{OT1}{cmss}{m}{n} {<5> <6> <7> <8> <9> <10> <11> <12> <13> <14.4> cmss10}{}
\DeclareMathAlphabet{\cmss}{OT1}{cmss}{m}{n}

\newtheoremstyle{thm}{1.5ex}{1.5ex}{\itshape\rmfamily}{} {\bfseries\rmfamily}{}{2ex}{}

\newtheoremstyle{def}{1.5ex}{1.5ex}{\rmfamily\sl}{} {\bfseries\rmfamily}{}{2ex}{}

\newtheoremstyle{rem}{1.3ex}{1.3ex}{\rmfamily}{} {\bfseries\rmfamily}{}{2ex}{}

\newtheoremstyle{ass}{1.5ex}{1.5ex}{\rmfamily\sl}{} {\bfseries\rmfamily}{}{2ex}{}

\newenvironment{proofsect}[1] {\vskip0.1cm\noindent{\rmfamily\itshape#1.}}{\qed\vspace{0.15cm}}

\theoremstyle{thm}
\newtheorem{theorem}{Theorem}[section]
\newtheorem{lemma}[theorem]{Lemma}
\newtheorem{proposition}[theorem]{Proposition}

\newtheorem*{Main Theorem}{Main Theorem.}

\newtheorem{conjecture}[theorem]{Conjecture}

\newtheorem{assumption}{Assumption}

\theoremstyle{def}

\theoremstyle{rem}
\newtheorem{remark}[theorem]{{Remark}}

\numberwithin{equation}{section}

\renewcommand{\section}{\secdef\sct\sect}
\newcommand{\sct}[2][default]{\refstepcounter{section}
\addcontentsline{toc}{section}
{{\tocsection {}{\thesection}{\!\!\!\!#1\dotfill}}{}}
\vspace{0.7cm}
\centerline{ 
\scshape\arabic{section}.\ #1} \nopagebreak \vspace{0.2cm}}
\newcommand{\sect}[1]{
\vspace{0.4cm} \centerline{\large\scshape\rmfamily #1}
\vspace{0.2cm}}

\renewcommand{\subsection}{\secdef\subsct\sbsect}
\newcommand{\subsct}[2][default]{\refstepcounter{subsection}
\addcontentsline{toc}{subsection}
{{\tocsection{\!\!}{\hspace{1.2em}\thesubsection}{\!\!\!\!#1\dotfill}}{}}
\nopagebreak\vspace{0.45\baselineskip} {\flushleft\bf
\arabic{section}.\arabic{subsection}~\bf #1.~}
\\*[3mm]\noindent
\nopagebreak}
\newcommand{\sbsect}[1]{\vspace{0.1cm}\noindent
\textbf{#1.~}\vspace{0.1cm}}

\renewcommand{\subsubsection}{%
\secdef \subsubsect\sbsbsect}
\newcommand{\subsubsect}[2][default]{%
\refstepcounter{subsubsection} 
\addcontentsline{toc}{subsubsection}{{\tocsection{\!\!}
{\hspace{3.05em}\thesubsubsection}{\!\!\!\!#1\dotfill}}{}}
\nopagebreak
\vspace{0.15\baselineskip} \nopagebreak {\flushleft\rmfamily
\itshape\arabic{section}.\arabic{subsection}.\arabic{subsubsection}
\ \rmfamily #1\/.}\ }
\newcommand{\sbsbsect}[1]{\vspace{0.1cm}\noindent
\rmfamily \itshape
\arabic{section}.\arabic{subsection}.\arabic{subsubsection} \
\sffamily #1\/.\ }

\renewcommand{\caption}[1]{%
\vglue0.5cm
\refstepcounter{figure}
\begin{minipage}{0.9\textwidth}\small {\sc Figure~\thefigure. }#1\end{minipage}}

\newcommand{\textd}{\text{\rm d}\mkern0.5mu}

\newcommand{\texte}{\text{\rm  e}\mkern0.7mu}

\newcommand{\Cov}{\text{\rm \Cov}}
\newcommand{\1}{{1\mkern-4.5mu\textrm{l}}}
\renewcommand{\1}{\text{\sf 1}}

\renewcommand{\AA}{\mathcal A}

\newcommand{\CC}{\mathcal C}

\newcommand{\EE}{\mathcal E}
\newcommand{\FF}{\mathcal F}
\newcommand{\GG}{\mathcal G}

\newcommand{\B}{\mathbb B}

\newcommand{\E}{\mathbb E}

\newcommand{\N}{\mathbb N}

\newcommand{\BbbP}{\mathbb P}
\newcommand{\Q}{\mathbb Q}
\newcommand{\R}{\mathbb R}

\newcommand{\Z}{\mathbb Z}

\newcommand{\scrC}{\mathscr{C}}

\newcommand{\scrF}{\mathscr{F}}

\newcommand{\cc}{{\text{\rm c}}}

\def\myffrac#1#2 in #3{\raise 2.6pt\hbox{$#3 #1$}\mkern-1.5mu\raise 0.8pt\hbox{$#3/$}\mkern-1.1mu\lower 1.5pt\hbox{$#3 #2$}}
\newcommand{\ffrac}[2]{\mathchoice%
	{\myffrac{#1}{#2} in \scriptstyle}
	{\myffrac{#1}{#2} in \scriptstyle}
	{\myffrac{#1}{#2} in \scriptscriptstyle}
	{\myffrac{#1}{#2} in \scriptscriptstyle}
}

\newcommand{\ul}{\underline}

\newcommand{\hatX}{\hat{X}}

\begin{document}

\title[Trapping in the random conductance model]
{\large Trapping in the random conductance model}
\author[M.~Biskup, O.~Louidor, A.~Rozinov and A.~Vandenberg-Rodes]
{M.~Biskup$^{1,2}$, O.~Louidor$^{1}$, A.~Rozinov$^{1,3}$, A.~Vandenberg-Rodes~$^{1,4}$}
\thanks{\hglue-4.5mm\fontsize{9.6}{9.6}\selectfont\copyright\,\textrm{2012} \textrm{M.~Biskup, O.~Louidor, A.~Rozinov and A.~Vandenberg-Rodes.
Reproduction, by any means, of the entire
article for non-commercial purposes is permitted without charge.\vspace{2mm}}}
\maketitle

\vspace{-5mm}
\centerline{\textit{$^1$Department of Mathematics, UCLA, Los Angeles, California, USA}}
\centerline{\textit{$^2$School of Economics, University of South Bohemia, \v Cesk\'e Bud\v ejovice, Czech Republic}}
\centerline{\textit{$^3$Courant Institute, NYU, New York City, New York, USA}}
\centerline{\textit{$^4$Department of Mathematics, UCI, Irvine, California, USA}}

\vspace{2mm}
\begin{quote}
\footnotesize \textbf{Abstract:}
We consider random walks on~$\Z^d$ among nearest-neighbor random conductances which are 
i.i.d., positive, bounded uniformly from above but whose support extends all the way to zero. Our focus is on the detailed properties of the paths of the random walk conditioned 
to return back to the starting point at time $2n$. We show that in the situations when the 
heat kernel exhibits subdiffusive decay --- which is known to occur in dimensions $d\ge4$ 
--- the walk gets trapped for a time of order~$n$ in a small spatial region. 
This shows that the strategy used earlier to infer subdiffusive lower bounds on the heat kernel in specific examples is in fact dominant. In addition, we settle a conjecture concerning the worst possible subdiffusive decay in four dimensions. 
\end{quote}
\vspace{-4mm}

\section{Introduction}
\vspace{-6mm}
\subsection{Motivation}\noindent
Random walks among random conductances (a.k.a.\ the Random Conductance Model) are among the best studied examples of random walks in random environments. Indeed, it was in this context where the first general cases of an (\emph{annealed}) invariance principle were established (Kipnis and Varadhan~\cite{Kipnis-Varadhan}) and the decay of transition probabilities characterized (Delmotte~\cite{Delmotte}). Recently, the model enjoyed another wave of concerted work whose motivation stemmed from several sources. First, the invariance principle --- for elliptic cases~\cite{Kipnis-Varadhan} or not (De Masi, Ferrari, Goldstein and Wick~\cite{DMFGW1,DMFGW2}) --- required averaging the path law over the environment; the desire was to remove this averaging and prove a \emph{quenched} version of the result. Second, the analytic methods employed in \cite{Delmotte} hinged upon the assumption of uniform ellipticity and it was unclear how to proceed in the absence thereof.

As it turns out, both problems were resolved roughly at the same time and using similar methods (Sidoravicius and Sznitman~\cite{Sidoravicius-Sznitman}, Berger and Biskup~\cite{Berger-Biskup}, Mathieu and Piatnitski~\cite{Mathieu-Piatnitski}). A key input was to invoke (and obtain) a diffusive bound on the probability that the random walk is back to the starting point after a long time. This could be done in various specific cases of interest (e.g., for elliptic conductances by Delmotte~\cite{Delmotte}) including the random walk on the supercritical percolation cluster (Mathieu and Remy~\cite{Mathieu-Remy}, Barlow~\cite{Barlow}). However, instances of the Random Conductance Model have also been found --- in dimensions $d\ge5$ by Berger, Biskup, Hoffman and Kozma~\cite{BBHK}, Boukhadra~\cite{B1,B2} and in $d=4$ by Biskup and Boukhadra~\cite{Biskup-Boukhadra} --- where this probability decays \emph{subdiffusively}. Remarkably, this happens while a non-degenerate invariance principle holds for the paths (Mathieu~\cite{Mathieu-CLT}, Biskup and Prescott~\cite{Biskup-Prescott}). The two popular characterizations of ``diffusive behavior'' of the random walk employed in physics --- one based on the mean-square displacement of the $n$-th step of the walk and the other relying on the decay exponent of the return probabilities --- thus yield conflicting conclusions.

In the present note we will further elucidate the above phenomenon by analyzing the typical behavior of the random walk path conditioned to return to the starting point at a given time. Our main finding is that, whenever the return probabilities decay subdiffusively, the trapping strategies employed in \cite{BBHK,Biskup-Boukhadra,B1,B2} --- explicitly, that the walk gets stuck in a very small region for a positive fraction of its time --- are actually dominant. This may seem akin to the behavior seen in the Random Conductance Model with non-integrable upper tails; the limiting behavior there is described by the \emph{fractional kinetic equation} which corresponds to a Brownian path parametrized by the inverse of a stochastic (subordinator) process (Barlow and \v Cern\'y~\cite{Barlow-Cerny}, \v Cern\'y~\cite{Cerny}). However, in our case the trapping occurs at a bounded number of ``space-time'' locations, while for the fractional kinetic model it occurs at multiple scales and along the entire path.

\subsection{Model and known facts}
In order to keep further discussion more focused, let us introduce some notation. The random walks we will consider  invariably take place on the hypercubic lattice~$\Z^d$. Only nearest-neighbor transitions will be permitted with their probabilities given as follows: Let $(\omega_{xy})$ be the collection of positive numbers, called \emph{conductances}, that are indexed by unordered pairs of nearest neighbor vertices; i.e., $\omega_{xy}=\omega_{yx}$. The ``walk'' is actually a Markov chain $(X_n)$ with transition kernel
\begin{equation}
\cmss P_\omega(x,y):=\frac{\omega_{xy}}{\pi_\omega(x)}\quad\text{where}\quad\pi_\omega(x):=\sum_{z\colon|z-x|=1}\omega_{xz}.
\end{equation}
We will henceforth assume that the $\omega$'s are i.i.d.\ random variables with common law denoted by~$\BbbP$ and expectation by~$\E$. We will assume that the conductances are bounded from above, say, $\BbbP(\omega_{xy}\le1)=1$, but not away from zero. Notwithstanding, we impose $\BbbP(\omega_{xy}=0):=0$ throughout the paper to keep (many) calculations at a comfortable level.

Let us use $P_\omega^x$ to denote the law of the path ($X_n)$ subject to the initial condition $P_\omega^x(X_0=x) =1$. Clearly, $P_\omega^x(X_n=y)$ can also be written as the $n$-fold product $\cmss P^n_\omega(x,y)$ of the transition kernel evaluated between~$x$ and~$y$. Central to our attention is the precise decay of the \emph{diagonal} term, $\cmss P_\omega^{2n}(0,0)$, as $n\to\infty$. A quenched invariance principle is valid in our setting (Mathieu~\cite{Mathieu-CLT}, Biskup and Prescott~\cite{Biskup-Prescott}) and one thus immediately has the lower bound
\begin{equation}
\label{E:diffuse}
\cmss P_\omega^{2n}(0,0)\ge\frac{C(\omega)}{n^{d/2}},\qquad n\ge1,
\end{equation}
for some $\BbbP$-a.s.\ positive~$C(\omega)$; see e.g.\ Lemma~5.1 in Biskup~\cite{Biskup-review}. This rules out a superdiffusive scaling and, naturally, leads one to bet on a diffusive behavior. However, attempts to prove a corresponding upper bound failed due to the fact that all methods known for this purpose require some level of uniform ellipticity (which is simply not there for the cases under consideration). As it turned out, these efforts had no chance of succeeding --- such upper bounds actually do not hold despite the non-degenerate diffusive scaling of the entire path.

It was Fontes and Mathieu~\cite{Fontes-Mathieu} who first raised doubts about the general validity of diffusive heat-kernel upper bounds by constructing a law on conductances (not exactly i.i.d., but close enough), in all $d\ge1$, for which the \emph{expectation} $\E\cmss P_\omega^{2n}(0,0)$ decays arbitrarily slowly. However, as we are dealing with tail probabilities, it was not clear how much of this is the effect of averaging.

Motivated by this, Berger, Biskup, Hoffman and Kozma~\cite{BBHK} turned to the study of the quenched decay. It is not hard to check that, in $d=1$, the decay can be arbitrarily slow, so the best general estimate one can hope for is, in this case, $\cmss P_\omega^{2n}(0,0)=o(1)$. Above one dimension, the following general upper bounds were obtained:
\begin{equation}
\label{E:me}
\cmss P_\omega^n(0,0)\le C(\omega)\,
\begin{cases}
n^{-d/2},\qquad&d=2,3,
\\
n^{-2}\log n,\qquad&d=4,
\\
n^{-2},\qquad&d\ge5,
\end{cases}
\end{equation}
where $C(\omega)<\infty$ $\BbbP$-a.s. This matches the lower bound \eqref{E:diffuse} in dimensions $d=2,3$ (and thus shows that averaging \emph{was} the culprit in this case) but leaves a gap in all $d\ge4$. 

This point appeared to be a good time to start searching for possible counterexamples. And, indeed, soon enough an (i.i.d., nearest-neighbor) environment~$\BbbP$ was constructed, for any sequence $\lambda_n\to\infty$ and any $d\ge2$, such that
\begin{equation}
\label{E:lower}
\cmss P_\omega^{2n}(0,0)\ge\frac{C(\omega)}{n^2\lambda_n},
\end{equation}
along a deterministic subsequence $n_k\to\infty$; cf~\cite[Theorem~9]{BBHK}. Note that this decays faster than $n^{-d/2}$ in $d=2,3$ but slower in $d\ge5$. Hence \eqref{E:me} seemed
to be right on target (at least for $d \neq 4$ but, as we will show here, also for $d = 4$).

The gist of the construction of such an environment is simple to describe: For a positive integer~$n$, call an edge $\langle x,y\rangle$ a \emph{trap of scale~$n$} if $\omega_{xy}=1$ and $\omega_{xz}=\ffrac1n$ for all neighbors~$z\ne y$ of~$x$, and similarly $\omega_{yz}=\ffrac1n$ for all neighbors $z\ne x$ of~$y$. Call a trap $\lambda_n$-accessible if there is a path of length $o(\log\lambda_n)$ of edges with conductance one that connects a neighbor of~$x$ or~$y$ to a fixed ($\omega$-dependent) neighborhood of the origin. Now construct~$\BbbP$ so that a typical configuration will contain a $\lambda_n$-accessible trap of scale~$n$ for all sufficiently large~$n$ in a (sparse) deterministic sequence tending to infinity.  

The bound \eqref{E:lower} is then the result of the following strategy: The walk finds the path of conductance one in a finite number of steps and then travels along it towards the trap. Then it jumps across a~$\ffrac1n$-edge into the trap --- paying order~$\ffrac1n$ in probability --- after which it is happy to keep bouncing back and forth on $\langle x,y\rangle$ for any given time of order~$n$ (note that the escape probability is order $\ffrac 1n$). Then we make it emerge from the trap by crossing the same $\ffrac1n$-edge (paying again order~$\ffrac1n$) just in time for it to make it back, backtracking its own steps, to the origin in total time~$2n$. Using the Markov property, this gives
\begin{equation}
\cmss P_\omega^{2n}(0,0)\ge C(\omega)\,\texte^{-o(\log\lambda_n)}\,\frac1n\,\texte^{-O(1)}\,\frac1n\,\texte^{-o(\log\lambda_n)},
\end{equation}
which easily yields \eqref{E:lower}. A more detailed analysis (Boukhadra~\cite{B1,B2}) reveals that the transition between diffusive and subdiffusive regimes occurs in the class of power-law tails. 

As can be expected, in $d=4$ the corresponding construction becomes considerably more difficult, but even here one can find~$\BbbP$, for any~$\lambda_n\to\infty$, such that
\begin{equation}
\label{E:lower2}
\cmss P_\omega^{2n}(0,0)\ge C(\omega)\frac{\log n}{n^2\lambda_n},
\end{equation}
along a deterministic sequence of~$n$'s tending to infinity. The appearance of $\log n$ is due to the fact that here  trapping occurs (roughly) uniformly likely along a sequence of exponentially growing spatial scales; cf Biskup and Boukhadra~\cite{Biskup-Boukhadra} for further details.

We remark that subdiffusive decay of return probabilities has recently been demonstrated also in dynamical random environments by Buckley~\cite{Buckley}. Interestingly, the variable nature of the environment permits  one to achieve even decays close to $n^{-1}$. This is (roughly) because one can arrange the environment so that the walk has to pay to enter the trap, but exits it freely. It is not clear whether the $n^{-1}$-decay is the worst one can get under reasonable mixing assumptions on the environment dynamics.

As a final note we wish to reiterate that all of the subdiffusive decays occur while the quenched invariance principle holds for the path law. Thanks to very interesting observations made by Delmotte and Rau~\cite{Delmotte-Rau}, also  other characteristics of the law (e.g., the expected time to exit a large ball) behave as for uniformly elliptic environments.

\subsection{Problems to be addressed here}
Although the lower bounds \eqref{E:lower} and \eqref{E:lower2} and the upper bounds \eqref{E:me} are quite similar, they do not match each other completely because of the requirement~$\lambda_n\to\infty$. In $d\ge5$ this gap is closed by another result from \cite{BBHK} which states
\begin{equation}
\label{E:1.7}
\text{in }d\ge5:\qquad
n^2\,\cmss P_\omega^{2n}(0,0)\,\underset{n\to\infty}\longrightarrow\,0,\qquad \BbbP\text{-a.s.}
\end{equation}
Unfortunately, the argument in \cite{BBHK} does not extend to $d=4$ and so, even in the presence of examples satisfying the lower bound \eqref{E:lower2}, the story is not entirely finished in $d=4$. This is one of the problems to be resolved in this note (Theorem~\ref{thm:d=4}).

Another question of interest concerns the behavior of the paths that carry $P_\omega^0(X_{2n}=0)$. For ``regularly behaved'' random walks we expect, somewhat tautologically, that the path scales, as~$n\to\infty$, to a Brownian bridge. However, once $\cmss P^{2n}(0,0)$ decays subdiffusively, this can no longer be true. The fact that the lower bounds and the upper bounds can be matched suggests a possibility of a path getting stuck for a time of order~$n$ at a particular (spatially small) location --- a trap. However, it was not known whether multiple (more than a few) visits to traps or some other strategies cannot do even better. The main contribution in this paper (Theorems~\ref{thm:trapping} and \ref{thm:enhanced})
is an answer to this question, namely a rather precise description of the typical trapping strategy 
employed by the random walk in order to achieve subdiffusivity.

\section{Assumptions and Results}
\noindent
Although we treat here only the case of i.i.d.\ conductances, there are in fact only a few specific facts about the environment that we use in the proofs. These may be satisfied by other conductance laws and for this reason we explicitly state in this section all the assumptions which we rely upon later. Additional examples of environment distributions for which our results hold are discussed after the statement of the main results.

\subsection{Setup and assumptions}
Let $\B(\Z^d)$ denote the set of unordered (nearest-neighbor) edges in~$\Z^d$ and let $\Omega:=(0,1]^{\B(\Z^d)}$ be the set of allowed conductance configurations. We endow~$\Omega$ with the usual product $\sigma$-algebra. We will write $\langle x,y\rangle$ for the edge with endpoints $x,y\in\Z^d$; since the edges are not oriented, we have $\langle x,y\rangle=\langle y,x\rangle$. In particular, if $\omega\in\Omega$, we can interchangeably write
\begin{equation}
\omega_{xy}=\omega_{yx}=\omega_b,\quad\text{where}\quad b:=\langle x,y\rangle\in\B(\Z^d).
\end{equation}
On $\Omega$, there is a natural notion of a ``shift-by~$x$,'' denoted by $\tau_x$, which is defined by
\begin{equation}
(\tau_x\omega)_{y,z}:=\omega_{y+x,z+x},\qquad x,y,z\in\Z^d.
\end{equation}
Notice that the origin in environment $\tau_x\omega$ corresponds to vertex~$x$ in environment~$\omega$.
Our first assumption pertains directly to the law of $\omega$:

\begin{assumption}
\label{A1}
The environment law $\BbbP$ is a probability measure on (the product $\sigma$-algebra of)~$\Omega$ which is (jointly) ergodic with respect to the shifts $(\tau_x)_{x\in\Z^d}$.
\end{assumption}

The relevance of the ergodicity assumption is seen from the fact that, for any sequence $a_n\to\infty$ and $\star\in\{0,\infty\}$, the events
\begin{equation}
\AA^+_\star:=\bigl\{\omega\colon\limsup_{n\to\infty}\,a_n\cmss P_\omega^{2n}(0,0)=\star\bigr\}
\end{equation}
and
\begin{equation}
\AA^-_\star:=\bigl\{\omega\colon\liminf_{n\to\infty}\,a_n\cmss P_\omega^{2n}(0,0)=\star\bigr\}
\end{equation}
are shift-invariant and thus zero-one for any ergodic~$\BbbP$. To see why shift-invariance holds, we recall that the diagonal heat-kernel is non-increasing; cf~\cite[Lemma~3.9]{BBHK}:
\begin{equation}
\label{E:2.5a}
\cmss P_\omega^{2n}(x,x)\ge\cmss P_\omega^{2n+2}(x,x),\qquad n\ge0.
\end{equation}
This implies $\cmss P_\omega^{2n}(x,x)\ge\cmss P_\omega^{2n+2}(x,x)\ge C(\omega)\cmss P_\omega^{2n}(y,y)$, whenever $x$ and~$y$ are neighbors on~$\Z^d$, with $C(\omega):=\cmss P_\omega(x,y)\cmss P_\omega(y,x)
> 0$. As $\cmss P_\omega^{2n}(x,x)=\cmss P_{\tau_x\omega}^{2n}(0,0)$, we are done.

\smallskip
Our results and their proofs require making a (somewhat arbitrary) distinction between strong and weak edges. This will be done by introducing a positive cut-off $\alpha$ and calling edges $b$ with $\omega_b\ge\alpha$ strong and the others weak. If (as in the i.i.d.\ case), for $\alpha > 0$ small enough, the strong edges form an infinite connected component --- the \emph{strong component} $\scrC_{\infty,\alpha}$ --- whose complement has only finite connected components, we may choose to observe the random walk only when it is on~$\scrC_{\infty,\alpha}$. This defines a coarse-grained random walk $\hatX$. 
This walk is again a Markov chain, but now with the states restricted to $\scrC_{\infty,\alpha}$. The benefits of considering the coarse-graining are twofold: First, we will be able to employ arguments which require the conductances to be bounded uniformly away from zero. Second, the coarse-grained setting provides a natural, and completely geometric, approach to the notion of trapping. 

\smallskip
Let us thus pick an~$\alpha>0$ and let $\B_\alpha(\omega):=\{b\in\B(\Z^d)\colon \omega_b\ge\alpha\}$. Let $\scrC_{\infty,\alpha}=
\scrC_{\infty,\alpha}(\omega)$ denote the set of vertices in~$\Z^d$ that lie on an infinite self-avoiding path of edges from~$\B_\alpha(\omega)$. We will often regard $\scrC_{\infty, \alpha}$ as a sub-graph of $\Z^d$ with the edge set $\{\langle x,y\rangle\in \B_\alpha(\omega)\colon x,y\in\scrC_{\infty,\alpha}\}$. Our first structural assumption is then:

\begin{assumption}
\label{A2}
For~$\BbbP$-a.e.~$\omega$,
\begin{equation}
\label{eqn:A2}
\bigcup_{\alpha > 0} \scrC_{\infty, \alpha}(\omega)=\Z^d.
\end{equation}
\end{assumption}

Notice that Assumption~\ref{A2} implies
\begin{equation}
\label{eqn:alpha1}
\alpha_1:=\inf\bigl\{\alpha>0\colon\BbbP(\scrC_{\infty,\alpha}\ni0)=0\bigr\}>0.
\end{equation} 
For $\alpha\in(0,\alpha_1)$, we may thus define the conditional measure
\begin{equation}
\BbbP_\alpha(-):=\BbbP(-|0\in\scrC_{\infty,\alpha}).
\end{equation}
We remark that, for general environments, the graph $\scrC_{\infty,\alpha}$ may not be connected and, in fact, uniqueness of the infinite cluster of strong edges is not required. Notwithstanding, uniqueness 
certainly holds in the i.i.d.\ case (e.g., by Burton and Keane~\cite{Burton-Keane}).

Our next concern are the connected components of the set~$\Z^d\setminus\scrC_{\infty,\alpha}$ --- with connectedness induced by the graph structure of the full~$\Z^d$. For~$y\in\Z^d\setminus\CC_{\infty,\alpha}$, let $\FF_y=\scrF_y(\omega)$ denote the connected component of $\Z^d\setminus \scrC_{\infty, \alpha}(\omega)$ containing~$y$; for $y\in\CC_{\infty,\alpha}$ we set $\FF_y:=\{y\}$. Define $\GG_x=\GG_x(\omega)$ by
\begin{equation}
\label{eqn:Pa}
\GG_x:=\bigcup_{y\colon\langle x,y\rangle\in\B(\Z^d)}\FF_y. 
\end{equation}
We will use $|\GG_x|$ to denote the cardinality of~$\GG_x$.

\begin{assumption}
\label{A3}
There is~$\alpha \in(0,\alpha_1)$ such that
\begin{equation}
\label{E:2.10}
|\GG_0|\in L^p(\BbbP_{\alpha}),\qquad \forall\ 1\le p<\infty.
\end{equation}
\end{assumption}

Note that \eqref{E:2.10} implies that $\max_{|x|\le n}|\GG_x|=n^{o(1)}$ as $n\to\infty$, $\BbbP$-a.s.
Also, it is not hard to check that, as $\alpha$ decreases, $\scrC_{\infty,\alpha}(\omega)$ increases 
and $\GG_0(\omega)$ decreases. Therefore, \eqref{E:2.10} is a monotone (in $\alpha$) property and we
accordingly define
\begin{equation}
\label{eqn:alpha2}
\alpha_2:=\sup \bigl\{\alpha \in (0, \alpha_1) \colon \text{ \eqref{E:2.10} holds}\} \,.
\end{equation} 

For the discussion to come next, let $\omega\in\Omega$ be such that $\Z^d\setminus\scrC_{\infty,\alpha}(\omega)$ has only finite components. We will now define the aforementioned coarse-grained walk~$\hatX$. First we record the times that the walk~$X$ spends away from~$\scrC_{\infty,\alpha}$. Let us set $T_0:=0$ and define
\begin{equation}
T_{k+1}:=\inf\bigl\{n>T_0+\dots+T_k\colon X_n\in\scrC_{\infty,\alpha}\bigr\}-(T_0+\dots+T_k),\qquad k\ge0.
\end{equation}
The quantity $T_k$ --- which is finite $P_\omega^x$-a.s.\ for all~$x$ and all~$k\ge1$ --- is the time between the $(k-1)$-st and $k$-th visit to~$\scrC_{\infty,\alpha}$. These visits occur at the locations
\begin{equation}
\hatX_\ell:=X_{T_0+\dots+T_\ell},\qquad \ell\ge0.
\end{equation}
The sequence $(\hat X_\ell)$ is a Markov chain on~$\scrC_{\infty,\alpha}$  whose transition kernel is  given by
\begin{equation}
\hat{\cmss P}_\omega(x,y):=P_\omega^x(X_{T_1}=y).
\end{equation}
It is easy to verify that both $\cmss P_\omega$ and $\hat{\cmss P}_\omega$ are reversible with respect to the measure~$\pi_\omega$ (on $\Z^d$ and ~$\scrC_{\infty,\alpha}$ respectively).
Denoting by $\hat{\cmss P}_\omega^n$ the $n$-fold product of $\hat{\cmss P}_\omega$, define
\begin{equation}
\textd_\omega(x,y):=\inf\bigl\{n\ge0\colon\hat{\cmss P}_\omega^n(x,y)>0\bigr\},\qquad x,y\in\scrC_{\infty,\alpha},
\end{equation}
to be the Markov distance (metric) on~$\scrC_{\infty,\alpha}$. As our final assumption, we postulate a uniform diffusive upper bound on the $n$-step transition probability of the coarse-grained walk.

\begin{assumption}
\label{A5}
There is~$\alpha_0\in(0,\alpha_2]$ such that for each~$\alpha\in(0,\alpha_0)$ and each $\rho > 0$, 
there is a $\BbbP_\alpha$-a.s.\ finite random variable~$C=C(\omega)$ such that for $\BbbP_\alpha$-a.e.~$\omega$, 
\begin{equation}
\label{eqn:2.13}
\max_{\begin{subarray}{c}
x\in\scrC_{\infty,\alpha}\\\textd_\omega(0,x)\le \rho n
\end{subarray}}\,\,\sup_{y\in\scrC_{\infty,\alpha}}
\hat{\cmss P}_\omega^n(x,y)\le\frac{C(\omega)}{n^{d/2}},\qquad n\ge1.
\end{equation}
\end{assumption}

By reversibility, \eqref{eqn:2.13} holds also with $\hat{\cmss P}^n_\omega(y,x)$ instead of $\hat{\cmss P}^n_\omega(x,y)$ with a constant $C^\prime(\omega) \leq (2d/\alpha) C(\omega)$. 
The following proposition formally ensures that under the product law all above assumptions are satisfied.

\begin{proposition}
\label{lem:perc_satisfies_assumptions}
Any product law on~$\Omega$ obeys Assumptions~\ref{A1}-\ref{A5}.
\end{proposition}
\begin{proof}
The proof is essentially contained in \cite{Biskup-Prescott, BBHK} --- only minor modifications are required. 
Indeed, \ref{A1} holds as $\BbbP$ is clearly ergodic. \ref{A2} follows since $\scrC_{\infty, \alpha} \neq \emptyset$ for $\BbbP$-a.s.\ once $\alpha > 0$ is such that $p:=\BbbP(\omega_b \geq \alpha) > p_\cc(d)$, 
where $p_\cc(d)$ is the Bernoulli bond percolation threshold on $\Z^d$. (We are also using that all conductances are positive.) \ref{A3} is covered by Proposition~2.3  of \cite{Biskup-Prescott}. \ref{A5} for $\rho=1$ follows by combining (6.10) in Proposition~6.1 
of \cite{Biskup-Prescott} and Lemma~3.4 in \cite{BBHK}. (As decreasing $\alpha$ permits us to increase~$p$ arbitrarily close to one, the isoperimetric inequality for large sets is proved by the argument from Benjamini and Mossel~\cite[Section~2.4]{Benjamini-Mossel}.) The extension to general $\rho > 0$, requires 
only that we consider boxes of side length $(\rho+1)n$ instead of $2n$ in Proposition~6.1. This can be achieved by a slight reduction of exponent~$\nu$ in formula (6.5) of~\cite{Biskup-Prescott}; Lemma~3.4 in \cite{BBHK} then can be used as is. (We note that the control of isoperimetric volumes provided by \cite[Proposition~A.2]{BBHK} or Benjamini-Mossel's argument in fact yields \eqref{eqn:2.13} with~$\rho n$ replaced by a quantity that grows \emph{exponentially} with~$n^{\nu\frac d{d-1}}$.)
\end{proof}

It is worthy a note that Assumptions~\ref{A1}-\ref{A5} are sufficient to ensure a quenched invariance principle for the corresponding random walk. This follows from Biskup and Prescott~\cite{Biskup-Prescott}.

\subsection{Results}
We are now ready to state our main results. Our first task will be to close the gap between the upper bound in \eqref{E:me} for $d=4$ and the lower bound in \eqref{E:lower2}:

\begin{theorem}
\label{thm:d=4}
Let $d=4$. Then under Assumptions~\ref{A1}-\ref{A5}, for $\BbbP$-a.e.~$\omega$,
\begin{equation}
\cmss P^{2n}_\omega(0,0)=o(n^{-2} \log n),\qquad n\to\infty.
\end{equation}
\end{theorem}

This settles an open question that was left unanswered in~\cite{BBHK} and~\cite{Biskup-Boukhadra}. As we will see in Section~\ref{sec3}, the argument seamlessly yields also the proof of \eqref{E:1.7}.

The next set of results concerns trapping effects. As already mentioned, we will describe
these by means of the times~$T_j$ the walk~$X$ takes between successive 
visits to~$\scrC_{\infty,\alpha}$. Define
\begin{equation}
\ell_n:=\inf\Bigl\{m\ge1\colon\sum_{k=1}^m T_k \geq 2n\Bigr\},\qquad n\ge1,
\end{equation}
and for $1 \leq r \leq \ell$ and any $\theta\ge1$, consider the sets
\begin{equation}
\label{E:2.19d}
G_{\,\ell,r}(\theta):=\Biggl\{(t_1,\dots,t_\ell)\in\N^\ell\colon \min_{1\le i_1<\dots<i_r\le\ell}\sum_{\begin{subarray}{c}
k=1,\dots,\ell\\ k\not\in\{i_1,\dots,i_r\}
\end{subarray}}
t_k\le\theta\Biggr\}.
\end{equation}
Use these to define the event
\begin{equation}
\EE_{\omega,\alpha}^{(r)}(\theta,n):=\big\{r\le\ell_{n}\le\theta,\,(T_1,\dots,T_{\ell_{n}})\in G_{\,\ell_{n},r}(\theta)\bigr\}.
\end{equation}
which, we note, depends explicitly on $\omega$ and $\alpha$. This definition will 
be made clear once we state our first trapping result:

\begin{theorem}
\label{thm:trapping}
Let $d \geq 4$ and suppose that Assumptions~\ref{A1}-\ref{A5} hold. Set $r:=\lfloor \tfrac d2-1\rfloor$
and define
\begin{equation}
\lambda_n(\omega):=n^{d/2}\cmss P_\omega^{2n}(0,0)
\end{equation}
Then, for all $\alpha\in(0,\alpha_0)$, there is a $\BbbP_\alpha$-a.s.\ finite $C=C(\omega)$ such that for all $n\ge1$ and all~$\theta$ with
$1\le \theta\le n/2$ there is $n^\star\in\{n-\theta,\dots,n\}$ for which
\begin{equation}
\label{E:2.16}
P_\omega^0\bigl(\EE_{\omega,\alpha}^{(r)}(\theta,n^\star)\big|X_{2n^\star}=0\bigr)
\ge 1-\frac{C(\omega)}{\lambda_n(\omega)}
\begin{cases}
\frac n\theta\log(\ffrac n\theta),\qquad&\text{if }d=4,
\\*[3mm]
\bigl(\tfrac n\theta\bigr)^{d/2-1},\qquad&\text{if }d\ge5.
\end{cases}
\end{equation}
\end{theorem}

In order to interpret the statement note that, as soon as $\cmss P_\omega^{2n}(0,0)$ decays subdiffusively along a subsequence of $n$'s tending to infinity, we have $\lambda_n(\omega)\to\infty$ (along this subsequence) and so we can choose $\theta=\theta_n$ in such a way that $\theta_n=o(n)$ while the right-hand side of \eqref{E:2.16} tends to one. For the corresponding sequence of~$n^\star$'s, the event on the left then holds with high probability. Now, on $\EE_{\omega,\alpha}^{(r)}(\theta,n)$, by time~$2n$ the walk~$X$ makes at most~$\theta$ visits to~$\scrC_{\infty,\alpha}$ while spending all but~$\theta$ units of time in at most~$r$ components of~$\Z^d\setminus\scrC_{\infty,\alpha}$. If~$\theta=o(n)$, the pigeon-hole principle ensures that at least one of these components traps the walk for a time of order~$n$.

\smallskip
Our final theorem addresses one of the deficiencies of Theorem~\ref{thm:trapping}; namely, the fact that the conclusion concerns~$n^\star$ instead of~$n$:

\begin{theorem}
\label{thm:enhanced}
Let $d \geq 4$ and suppose that Assumptions~\ref{A1}-\ref{A5} hold. Set $r:=\lfloor \tfrac d2-1\rfloor$
and define
\begin{equation}
\label{eqn:Zeta}
\zeta_n(\omega):=
\begin{cases}
\displaystyle\frac{n^2}{\sqrt{\log n}}\,\cmss P_\omega^{2n}(0,0),\qquad&\text{if }d=4,
\\*[3mm]
n^{(\ffrac d4+1)}\,\cmss P_\omega^{2n}(0,0),\qquad&\text{if }5 \le d \le 8,
\\*[2mm]
n^3\,\cmss P_\omega^{2n}(0,0),\qquad&\text{if }d\ge9.
\end{cases}
\end{equation}
Then, for all $\alpha\in(0,\alpha_0)$, there is a 
$\BbbP_\alpha$-a.s.\ finite $C=C(\omega)$ such that for all $n\ge1$ and 
all $\theta$ with $1 \le \theta \le n$,
\begin{equation}
\label{E:2.16a}
P_\omega^0\bigl(\EE_{\omega,\alpha}^{(r)}(\theta,n)\big|X_{2n}=0\bigr)
\ge 1-\frac{C(\omega)}{\zeta_n(\omega)}
\begin{cases}
\displaystyle
\sqrt{\log(\ffrac n\theta)},
\qquad&\text{if }d=4,
\\
\displaystyle \left(\frac n\theta \right)^{(\ffrac d4-1)},\qquad&\text{if } 5 \le d \le 8,
\\*[3mm]
\displaystyle \frac{n}{\theta^{\ffrac{d}{4}-1}} + 1 ,\qquad&\text{if } d \ge 9.
\end{cases}
\end{equation}
\end{theorem}

As before, once $\zeta_n\to\infty$ along a subsequence of~$n$'s, we can choose $\theta=o(n)$ so that the right-hand side tends to one (along the subsequence). Note that this will be possible when the 
decay of $\cmss P_\omega^{2n}(0,0)$ is {\em sufficiently} slower than $n^{-d/2}$.

\smallskip
A second deficiency of Theorem~\ref{thm:trapping} is the inability to exclude the possibility of multiple trapping locations. This is only an issue in $d\ge6$ because $r=1$ for $d=4,5$. Unfortunately, we do not know how to overcome this even for a strongly subdiffusive decay.

\subsection{Discussion\label{subs:Discussion}}
\noindent
We will finish with a couple of remarks on the scope and extensions of the above results. First, both Theorem~\ref{thm:trapping} and \ref{thm:enhanced} admit a slightly stronger formulation. Namely, one can get \eqref{E:2.16} and \eqref{E:2.16a} for $\BbbP$-a.e.\ $\omega$ and \emph{all} $\alpha\in(0,\alpha_0)$ provided the constant $C(\omega)$ retains an explicit dependence on~$\alpha$. We emphasize that this does not follow automatically from the above as $\EE_{\omega,\alpha}^{(r)}(\theta,n)$ depends explicitly on~$\alpha$. To control the continuum of possible~$\alpha$'s, we use that, for $\BbbP$-a.e.\ $\omega$,
\begin{equation}
0<\alpha\le\alpha'\le\alpha_0\quad\Rightarrow\quad\EE_{\omega,\alpha}^{(r)}(\theta,n)\subseteq\EE_{\omega,\alpha'}^{(r)}(\theta,n).
\end{equation}
In addition, as long as $\GG_0(\omega)$ is finite --- which happens for all $\alpha\in(0,\alpha_0)$ $\BbbP$-a.s.\ --- the conclusion is not affected by the fact that the walk does not start on the infinite component.

As to the need for a choice of~$n^\star$ in Theorem~\ref{thm:trapping}, we point out that the proof actually tells us more. Indeed, writing the right-hand side of \eqref{E:2.16} as $1-q_n$, from \eqref{eqn:4.17} we have
\begin{equation}
\#\Bigl\{m\in\{n-\theta,\dots,n\}\colon P_\omega^0\bigl(\EE_{\omega,\alpha}^{(r)}(\theta,m)\big|X_{2m}=0\bigr)\le 1-\epsilon\Bigr\}\le \frac1\epsilon q_n(\theta+1)
\end{equation}
for any fixed~$\epsilon>0$. Hence, as soon as the error probability~$q_n$ tends to zero, trapping
occurs at all but an $o(1)$-fraction of times in $\{n-\theta, \dots, n\}$ (just choose $\epsilon:=\sqrt{q_n}$). 
We, in fact, believe the following:
\begin{conjecture}
There is $\BbbP_\alpha$-a.s.\ finite~$C=C(\omega)$ such that, for all~$\theta$ sufficiently close to~$n$, the bound \eqref{E:2.16} holds also for~$n^\star=n$. 
\end{conjecture}

\noindent
Ideas invoked in the proof of Theorem~\ref{thm:enhanced} may be handy here as they establish an explicit bound on how fast the quantity $P_\omega^0(\EE_{\omega,\alpha}^{(r)}(\theta,m)\big|X_{2m}=0)$ may oscillate with~$m$.

Another note we wish to make concerns the geometric size of the trapped regions. From Assumption~\ref{A3} we know that the largest component of~$\Z^d\setminus\scrC_{\infty,\alpha}$ that the walk can reach (and thus become trapped by) in time of order~$n$ is at most $n^{o(1)}$ in diameter. However, the strategies employed for the proofs of the lower bounds indicate that a typical trapping region may be of finite order in size, regardless of~$n$. It is an open question to prove or disprove this rigorously (for i.i.d.\ environments, to begin with).

Finally, we wish to remark that all our results extend to more general environment distributions. 
First, one may consider any ergodic distribution on $\Omega$ which stochastically dominates a product law; indeed, Assumptions~\ref{A1}-\ref{A5} still hold in this case. (This is obvious for~\ref{A2}-\ref{A3}; for~\ref{A5} one needs to note that the Benjamini-Mossel~\cite[Section~2.4]{Benjamini-Mossel} argument for the isoperimetric inequality on percolation cluster for~$p$ close to one extends to any law that dominates this percolation measure.)
Second, for i.i.d.\ conductances, we may soften the requirement that the conductances be strictly positive. In this case, we can no longer impose \eqref{eqn:A2} as Assumption~\ref{A2} and  we must instead require \eqref{eqn:alpha1} directly. 
The rest of the assumptions as well as the proofs remain almost the same, except that instead of~$\Z^d$
we need to use the (random) set of edges with positive conductances as the effective underlying graph. With this generalization, the results apply to any product law for which
$\BbbP(\omega_b > 0) > p_\cc(d)$, where~$p_\cc(d)$ is the bond-percolation threshold on $\Z^d$.

\section{Decay in four dimensions}
\label{sec3}
\noindent
In this section we prove Theorem~\ref{thm:d=4}. Although our exposition below is by and large self-contained, we welcome the reader to check Section~3.3 and Proposition~3.5 in~\cite{BBHK} which serves as a foundation for the present proof.
A key technical step underlying all derivations in \cite{BBHK} is the conditioning on the number of steps taken by the coarse-grained walk. Indeed, whenever $X_0\in\scrC_{\infty,\alpha}$, we can write $\{X_{2m}=0\}$ as the disjoint union
\begin{equation}
\{X_{2m}=0\} = \bigcup_{\ell=1}^{2m}\bigl\{\hat X_\ell=0,\,T_1+\dots+T_\ell=2m\bigr\}.
\end{equation}
Another important fact that we will use frequently is the monotonicity of the diagonal heat-kernel \eqref{E:2.5a}. With the help of these we can write
\begin{equation}
\label{E:3.1}
\begin{aligned}
\cmss P_{\omega}^{4n}(0,0) 
  & \leq  n^{-1}\sum_{m=n}^{2n}\cmss P_{\omega}^{2m}(0,0)\\
& \leq n^{-1}\sum_{\ell=1}^{4n}P_{\omega}^0(\hatX_\ell=0, \ T_1+\dots+T_\ell\geq 2n)\\
& \leq 2n^{-1} \sum_{\ell=1}^{4n}P_{\omega}^0 \Bigl(\,\hatX_\ell=0,\, \sum_{i\leq\lceil\ell/2\rceil} T_i\ge n\Bigr),
\end{aligned}
\end{equation}
where we notice that on the event when $T_1+\dots+T_\ell\ge2n$ either the sum over $i\le\lceil\ell/2\rceil$ or the sum over $\ell-\lceil\ell/2\rceil\le i\le\ell$ exceed~$n$. Reversibility then implies that the second sum has the same bound as the first one.

Next we need a version of Proposition~3.5 from \cite{BBHK} with an explicit term on the right:

\begin{lemma}
\label{lemma-3.1}
There exists a $\BbbP_\alpha$-a.s.\ finite random variable $C(\omega)$ such that
\begin{equation}
\label{E:3.2}
P_{\omega}^0 \Bigl(\,\hatX_\ell=0,\, \sum_{i\leq\lceil\ell/2\rceil} T_i\ge n\Bigr)
\le C(\omega)\,\frac{\ell^{-d/2}}n\, E_\omega^0\Bigl(\,\sum_{i\leq\lceil\ell/2\rceil} T_i\,;\,\sum_{i\leq\lceil\ell/2\rceil} T_i\ge n\Bigr).
\end{equation}
\end{lemma}

\begin{proofsect}{Proof}
By conditioning on $\hat X_{\lceil\ell/2\rceil}$ we get
\begin{equation}
P_{\omega}^0 \Bigl(\,\hatX_\ell=0,\, \sum_{i\leq\lceil\ell/2\rceil} T_i\ge n\Bigr)
=\sum_{x \in \scrC_{\infty,\alpha}} P_{\omega}^0 \Bigl(\,\hat X_{\lceil\ell/2\rceil}=x,\, \sum_{i\leq\lceil\ell/2\rceil} T_i\ge n\Bigr)\hat{\cmss P}^{\ell-{\lceil\ell/2\rceil}}(x,0).
\end{equation}
Reversibility and Assumption~\ref{A5} then tell us that 
\begin{equation}
 \hat{\cmss P}^{\ell-{\lceil\ell/2\rceil}}(x,0) = \frac{\pi_\omega(0)}{\pi_\omega(x)}\hat{\cmss P}^{\ell-{\lceil\ell/2\rceil}}(0,x)
  \leq \frac{2d}{\alpha} \ell^{-d/2} \,,
\end{equation}
uniformly for all~$x$. The desired bound then follows by summation over~$x$ and an application of Chebyshev's inequality.
\end{proofsect}

For what follows, we need to recall the notion of the ``point of the view of the particle''. Given $\omega$ and a sample~$X$ of the random walk, the sequence $(\tau_{X_n}\omega)$ represents the environments seen from the position of the random walk. As it turns out, this is a Markov chain on~$\Omega$ with a reversible, stationary measure 
\begin{equation}
\Q(-):=\frac1Z\pi_\omega(0)\BbbP(-),\quad\text{where}\quad Z:=\E \pi_\omega(0).
\end{equation}
As is well known (see, e.g.,~\cite[Section~2.1]{Biskup-review}) the chain started from this measure is ergodic.

A similar construction can be carried through also for the chain~$(\hat X, T)$. The stationary distribution for the sequence $(\tau_{\hat X_n}\omega)$ is now given by
\begin{equation}
\Q_\alpha(-):=\Q(-|0\in\scrC_{\infty,\alpha}).
\end{equation}
which is again stationary and reversible. Interpreting the chain using an induced shift on the space of trajectories (see, e.g., \cite[Lemma~3.3]{Berger-Biskup}), starting from~$\Q_\alpha$, the process 
$(\tau_{\hat X_n}\omega,\; T_n)_{n \geq 1}$ is stationary and ergodic. In addition $\Q_\alpha \sim \BbbP_\alpha$, i.e., $\Q_\alpha$ is equivalent to~$\BbbP_\alpha$, for every $\alpha\in(0,\alpha_0)$.
\begin{lemma}
\label{lemma-Z's}
Abbreviate
\begin{equation}
Z_\ell:=\frac1\ell\sum_{j=1}^{\ell} T_j,\qquad \ell\ge1.
\end{equation}
Then for~$\BbbP_\alpha$-almost every $\omega$ we have
\begin{equation}
\label{eqn:3.9}
Z_\ell\,\underset{\ell\to\infty}\longrightarrow\, Z_\infty:=\E_{\Q_\alpha} E_\omega^0 T_1,
\end{equation}
$P_\omega^0$-almost surely and in $L^1(P_\omega^0)$.
\end{lemma}

\begin{proofsect}{Proof}
We shall prove the almost-sure and $L^1$ convergence in \eqref{eqn:3.9} for $\Q_\alpha$-almost every $\omega$. Since $\Q_\alpha \sim \BbbP_\alpha$, this will be enough.
Consider therefore the joint stationary measure $\mu:=\Q_\alpha\otimes P_\omega^0$ on the space of environments and paths of the random walk. By \cite[Lemma~3.8]{BBHK} we have
\begin{equation}
\label{eqn:Lemma38}
  E_{\omega}^xT_1 \leq c_1|\GG_x|,	\qquad x \in \scrC_{\infty,\alpha}(\omega),
\end{equation}
for some~$c_1=c_1(d,\alpha)\in(0,\infty)$. In particular, $\E_{\Q_\alpha} E_\omega^0 T_1 <\infty$ because the component sizes have all moments by Assumption~\ref{A3}. The ergodicity of the Markov chain on the space of environments then tell us that $Z_\ell\to Z_\infty$, $\mu$-a.s.\ and in $L^1(\mu)$. However, this is not enough to prove convergence in $L^1(P_\omega^0)$ because almost sure and~$L^1$~convergence do not generally guarantee convergence of conditional expectations.

We thus proceed by a more explicit argument. Since $Z_\ell\to Z_\infty$ almost surely with respect to~$\mu$, and thus also with respect to~$P_\omega^0$, for~$\BbbP_\alpha$-a.e.~$\omega$, in order to infer $L^1(P_\omega^0)$-convergence, it suffices to show the convergence of the norms, i.e.,
\begin{equation}
E_\omega^0 Z_\ell \,\underset{\ell\to\infty}\longrightarrow\, Z_\infty,\qquad \BbbP_\alpha\text{-a.s.}
\end{equation}
By the Markov property and additivity of expectations, this is equivalent to proving this for the corresponding expectation of the sequence of random variables
\begin{equation}
Y_\ell:=\frac1\ell\sum_{j=0}^{\ell-1} E_{\tau_{\hat{X}_j}\omega}^0(T_1).
\end{equation}
Indeed, $E_\omega^0 Z_\ell=E_\omega^0 Y_\ell$ for each $\ell\ge1$ and $Y_\ell\to Z_\infty$, $\mu$-a.s.

We will show $E_\omega^0Y_\ell\to Z_\infty$ by invoking the Dominated Convergence Theorem, but for that end we need to exhibit a dominating random variable that lies in $L^1(P^0_\omega)$, for $\BbbP_\alpha$-a.e.~$\omega$. Define
\begin{equation}
W_\ell:=\frac1\ell\sum_{j=0}^{\ell-1} |\GG_{\hatX_j}|,\qquad \ell\ge1,
\end{equation}
and set
\begin{equation}
W^\star:=\sup_{\ell\ge1}W_\ell.
\end{equation}
From \eqref{eqn:Lemma38} we observe that $(0\le)Y_\ell\le c_1 W_\ell\le c_1 W^\star$ and so $W^\star$ can indeed be used to dominate the $Y_\ell$'s. We thus need to prove
\begin{equation}
\label{E:3.11}
W^\star\in L^1(P_\omega^0),\qquad \BbbP_\alpha\text{-a.s.}
\end{equation} 
By Assumption~\ref{A3}, $|\GG_0|\in L^p(\mu)$ for all $p\ge1$. Wiener's Dominated Ergodic Theorem (cf~Petersen~\cite[Theorem~1.16]{Petersen}) then implies $W^\star\in L^p(\mu)$ for all $p\ge1$ as well. From here \eqref{E:3.11} follows via Fubini's Theorem.
\end{proofsect}

\begin{remark}
In the above proof we used the following sequence of estimates:
\begin{equation}
\label{E:3.12}
E_\omega^0\Bigl(\,\sum_{j=1}^{\ell}T_j\Bigr)= E_\omega^0\Bigl(\,\sum_{j=0}^{\ell-1}E_{\tau_{\hat X_j}\omega}^0(T_1)\Bigr)
\le c_1E_\omega^0\Bigl(\,\sum_{j=0}^{\ell-1}|\GG_{\hat X_j}|\Bigr)\le C(\omega)\ell,\qquad\ell\ge1,
\end{equation}
for some $\BbbP_\alpha$-a.s.\ finite random variable $C=C(\omega)$. This bound was invoked in~\cite[Eq.~(3.45)]{BBHK}, but without a reference to the Dominated Ergodic Theorem for the proof of the last step. It appears that one needs more than just plain integrability of $|\GG_0|$ for the last inequality to hold.
\end{remark}

Now we are ready to establish the upper bound on the four-dimensional heat kernel:

\begin{proof}[Proof of Theorem~\ref{thm:d=4}]
We now claim that, for any $M> \E_{\Q_\alpha} E_\omega^0 T_1$,
\begin{equation}
\label{E:3.8}
\lim_{n\to\infty}\,\max_{1\le\ell\le n/M}\,\,
\frac1\ell\, E_\omega^0\Bigl(\,\,\sum_{i\leq\ell} T_i\,;\,\sum_{i\leq\ell} T_i\ge n\Bigr)=0,\qquad \BbbP\text{-a.s.}
\end{equation}
To show this, use the bound $\ffrac n\ell\ge M$ to derive
\begin{equation}
\begin{aligned}
\frac1\ell \,E_\omega^0\Bigl(\,\,\sum_{i\leq\ell} T_i\,;\,\sum_{i\leq\ell} T_i\ge n\Bigr)
&=E_\omega^0\bigl(Z_{\ell}\,;\, Z_{\ell}\ge\ffrac n\ell\bigr)
\\
&\le E_\omega^0|Z_{\ell}-Z_\infty|+E_\omega^0\bigl(Z_\infty;Z_{\ell}\ge M\bigr).
\end{aligned}
\end{equation}
By Lemma~\ref{lemma-Z's} and the choice of~$M$ both terms on the right tend to zero as $\ell\to\infty$, so given $\epsilon>0$ we can find $\ell_0$ so that the left-hand side is less than $\epsilon$ for all $\ell$ with $\ell_0\le\ell\le n/M$. But for $\ell\le\ell_0$ the limit of the left-hand side as $n\to\infty$ is zero by the fact that the expectation of $\sum_{i=1}^\ell T_i$ is finite.

In order to prove the claim in the theorem, set $M>\E_{\Q_\alpha} E_\omega^0 T_1$, recall \eqref{E:3.1} and split the last sum in this formula according to whether $n/\ell>M$ or not. Fix~$\epsilon>0$ and let $n_0$ be so large that, for all $n\ge n_0$ and all $1\le\ell\le n/M$, the expectation on the right of \eqref{E:3.2} is less than~$\epsilon\ell$ (this is possible by \eqref{E:3.8}). For the complementary set of~$(n,\ell)$ pairs we use instead that
\begin{equation}
E_\omega^0\Bigl(\,\sum_{i\leq\lceil\ell/2\rceil} T_i\,;\,\sum_{i\leq\lceil\ell/2\rceil} T_i\ge n\Bigr)\le C(\omega)\ell,
\end{equation}
as is implied by \eqref{E:3.12}. Putting this together, we get for $n\ge n_0\vee M\ell_0$,
\begin{equation}
\label{E:3.19}
\cmss P_\omega^{4n}(0,0)\le C(\omega)n^{-1}\biggl(\,\sum_{n/M\le\ell\le n}\frac{\ell^{1-d/2}}n+
\epsilon\sum_{\ell=1}^{n/M}\frac{\ell^{1-d/2}}n\biggr).
\end{equation}
The first term in the parentheses on the right is bounded by $n^{-1}\log M$, once~$M$ is sufficiently large, while the second term is at most $\epsilon n^{-1}\log n$. As~$\epsilon$ was arbitrarily small and
as monotonicity implies that \eqref{E:3.19} holds for $\cmss P_\omega^{4n+2}(0,0)$ as well, the claim
follows.
\end{proof}

\begin{remark}
Note that in $d\ge5$, the first sum on the right-hand side of \eqref{E:3.19} is order $M^{2-d/2}n^{-1}$ while the second sum is order $n^{-1}$. This gives another proof of \eqref{E:1.7}; this time allowing for an extension to $d=4$.
\end{remark}

\section{Trapping under subdiffusive decay}
\noindent
Here we will establish the trapping scenario as stated in Theorem~\ref{thm:trapping}. The main technical obstacle for us is that Lemma~\ref{lemma-3.1} gives a good estimate for the sum over $T_i$ \emph{exceeding} $n$, rather than the event that the sum is \emph{equal} to~$n$. This necessitates that in many calculations we sum over a range of~$n$'s which then invariably leads to results for an~$n^\star$ in this range, rather than~$n$ itself. In what follows we will consider only $d\ge4$ and fix $\alpha \in (0, \alpha_0)$, where $\alpha_0$ is as in Assumption~\ref{A5}. Consequently, the statements and in particular the random constants $C(\omega)$ in the expressions below depend on~$\alpha$, but to avoid clutter, we shall not
reflect this in the notation.

Our starting point is the following observation: Should the walk spend a majority of its time in a small number of weak components, the total number of coarse-grained steps must satisfy $\ell_n=o(n)$. 
A  quantitative form of this is:

\begin{proposition}
\label{prop-3.4}
There is a $\BbbP_\alpha$-a.s.\ finite random variable~$C=C(\omega)$ such that 
for all $n\ge1$ and all $\theta$ and $\Delta$ with $1\le\theta,\Delta\le\ffrac n2$,
\begin{equation}
\frac1{\Delta+1}\sum_{m=n-\Delta}^n P_\omega^0\bigl(\ell_m\ge\theta,\,X_{2m}=0\bigr)
\le\frac{C(\omega)}{n^{d/2}}\Bigl(\frac n\Delta\Bigr)
\begin{cases}
\log(\ffrac n\theta),\qquad&\text{if }d=4,
\\*[2mm]
\bigl(\tfrac n\theta\bigr)^{\ffrac d2-2},\qquad&\text{if }d\ge5.
\end{cases}
\end{equation}
\end{proposition}

\begin{proofsect}{Proof}
For $m\le n$ we have
\begin{equation}
P_\omega^0\bigl(\ell_m\ge\theta,\,X_{2m}=0\bigr)
=\sum_{\ell=\theta}^{2n}P_\omega^0\biggl(\,\sum_{i=1}^\ell T_i=2m,\,\hat X_\ell=0\biggr).
\end{equation}
Summing over the given range of~$m$'s and noting that $n-\Delta\ge\ffrac n2$ yields
\begin{equation}
\sum_{m=n-\Delta}^nP_\omega^0\bigl(\ell_m\ge\theta,\,X_{2m}=0\bigr)
\le\sum_{\ell=\theta}^{2n}P_\omega^0\biggl(\,\sum_{i=1}^\ell T_i\ge n,\,\hat X_\ell=0\biggr).
\end{equation}
The $\ell$-th term on the right hand side can be estimated using Lemma~\ref{lemma-3.1} and \eqref{E:3.12} to be less than $C(\omega) n^{-1}\ell^{1-d/2}$. The claim now follows by summation over~$\ell$.
\end{proofsect}

Our next task is to estimate the probability of the event that the collection of times $(T_1,\dots,T_{\ell_m})$ does not have the property that removal of~$r$ of them makes the sum small. Explicitly:

\begin{proposition}
\label{prop-3.5}
Suppose $d\ge4$ and $r \ge \lfloor \tfrac d2-1 \rfloor$. Recall the definition of $G_{\,\ell,r}(\theta)$ from \eqref{E:2.19d}. For $\BbbP_\alpha$-a.e. $\omega$  there is 
$C(\omega)<\infty$ such that for all $n\ge1$ and all $\theta$ and~$\Delta$ with $1\le\Delta,\theta\le\ffrac n2$,
\begin{equation}
\label{E:3.21}
\frac1{\Delta+1}\,\sum_{m=n-\Delta}^n P_\omega^0\Bigl(\ell_m\le\theta,\, X_{2m}=0,\,(T_1,\dots,T_{\ell_m})\not\in G_{\,\ell_m,r}(\theta)\Bigr)
\le\frac{C(\omega)}{n^{d/2}}\Bigl(\frac n\Delta\Bigr)\Bigl(\frac n\theta\Bigr)^{d/2-2}.
\end{equation}
\end{proposition}

The proof will proceed along similar lines as that of Proposition~\ref{prop-3.4} except that now our goal is to obtain a bound on the $\ell$-th term in the sum which is \emph{not} summable on~$\ell$. Indeed, only then the sum will be dominated by the terms $\ell\approx\theta$. (This is actually the reason why we need to take~$r$ dependent on dimension.) A novel point compared to the previous proof is the condition $(T_1,\dots,T_{\ell_m})\not\in G_{\,\ell_m,r}(\theta)$. We will again convert the probability into expectation as follows: Let $r\ge1$ and $\ell\ge r$ and consider the set of distinct $r$-tuples
\begin{equation}
\label{E:3.22}
I(r, \ell) := \bigl\{\ul{i}=(i_1,\, \dots,\, i_r) \,\colon 1 \leq i_1 < i_2 < \dots < i_r \leq \ell\bigr\}.
\end{equation}
Then
\begin{equation}
\label{E:4.15d}
(T_1,\dots,T_\ell)\not\in G_{\,\ell,r}(\theta)\,\,\,\&\,\,\,\sum_{i=1}^\ell T_i\ge n
\quad\Rightarrow\quad
\sum_{\ul{i} \in I(r+1,\ell)} \prod_{k=1}^{r+1} T_{i_k}\ge\frac{n \theta^r}{(r+1)!}.
\end{equation}
To see why this holds, recall \eqref{E:2.19d} to see that if $(T_1,\dots,T_\ell)\not\in G_{\,\ell,r}(\theta)$, then the sum of the $T_i$ with~$i$ skipping out any $r$-tuple of indices yields at least~$\theta$. Writing the sum over $\ul{i} \in I(r+1,\ell)$ as $1/(r+1)!$ times the sum over $r+1$ distinct indices and summing one index after the other, we thus get at least a factor~$\theta$ for each of the first~$r$ sums. The last sum is then unconstrained and it yields at least~$n$, in light of the second condition on the left of \eqref{E:4.15d}.

Ignoring for a moment the condition $X_{2m}=0$, we are thus naturally led to a multiparameter version of \eqref{E:3.12}:

\begin{lemma}
\label{lemma-3.6}
For any $r\ge1$,
\begin{equation}
\label{E:3.23}
\sup_{\ell\geq r}\, \frac{1}{\ell^r}\, E_\omega^0 \biggl(\,\sum_{\ul{i} \in I(r,\ell)} \prod_{k=1}^r T_{i_k} \biggr)<\infty\qquad\BbbP_\alpha\text{\rm-a.s.}
\end{equation}
\end{lemma}

\begin{proof}
Let us use $M^r_0 = M^r_0(\omega)$ to denote the supremum in the statement of the lemma, the $0$ subscript 
indicating the starting point of the random walk, later to be replaced by any initial position 
$x \in \scrC_{\infty,\alpha}$. We will prove a stronger statement, namely that $M^r_0$ is in 
$L^p(\Q_\alpha)=L^p(\BbbP_\alpha)$ for any $p \geq 1$ and any $\alpha\in(0,\alpha_0)$. This will be done by induction on $r$.

For $r=1$, we argue as in the proof of Lemma~\ref{lemma-Z's}. Indeed with
the notation there, 
\begin{equation}
M^1_0 := \sup_{\ell\ge1} E_\omega^0 Z_l = \sup_{\ell\ge1} E_\omega^0 Y_l \leq c_1 \sup_{\ell\ge1} E_\omega^0 W_l
  \leq c_1 E_\omega^0 W^\star.
\end{equation}
Now, since $|\GG_0|$ has all moments under $\Q_\alpha$, Wiener's Dominated Ergodic Theorem for
$|\GG_0|^p$ and Jensen's inequality imply that $W^\star$ is in $L^p(\mu)$ for any
$p \geq 1$. Invoking Jensen one more time we conclude that $E_\omega^0 W^\star$ is in $L^p(\Q_\alpha)$ for all $p\ge1$.

For the induction step $r\to r+1$, using the strong Markov property we may write
\begin{equation}
\begin{aligned}
\frac1{\ell^{r+1}}\, E_\omega^0 \biggl(\,\sum_{\ul{i} \in I(r+1,\ell)} \prod_{k=1}^r T_{i_k} \biggr) 
     & \leq \ell^{-1} E_\omega^0 \biggl(\,\sum_{j=1}^{\ell-r} T_j 
	\, \ell^{-r} E_{\tau_{\hat X_{j}} \omega}^0 \Bigl(\,\sum_{\ul{i} \in I(r,\ell-j)} \prod_{k=1}^r T_{i_k}\Bigr)\biggr)
	\\
	&
	\leq \ell^{-1} E_\omega^0 \Bigl(\,\sum_{j=1}^{\ell} T_j M_{\hat X_{j}}^r \Bigr) 
	= E_\omega^0 \Bigl(\ell^{-1} \sum_{j=0}^{\ell-1} E_{\tau_{\hat X_j}\omega}^0(T_1 M_{\hat X_1}^r) \Bigr). 
\end{aligned}
\end{equation}
Therefore, by the Dominated Ergodic Theorem and Jensen as in the base of the induction, $M^{r+1}_0$ 
will be in $L^p(\Q_\alpha)$ for any $p \geq 1$ once the same is true for $E_\omega^0(T_1 M_{\hat X_1}^r)$.

To show this, we write
\begin{equation}
\label{eqn:MaxBoundedBySum}
  E_\omega^0 (T_1 M_{\hat X_1}^r) \leq E_\omega^0(T_1) \sum_{x \in \partial \GG_0} M_x^r
    \leq c_1 |\GG_0| \sum_{|x|_\infty \leq \text{diam}(\GG_0)} M_x^r \1_{x \in \scrC_{\infty, \alpha}}.
\end{equation}
Now $|\GG_0|$, $\text{diam}(\GG_0)$ and $M_0^r$ are in $L^p(\Q_\alpha)$ for any $p \geq 1$; the first
two due to Assumption~\ref{A3}, the latter by the induction hypothesis. Denoting $N:=\sum_{|x|_\infty \leq \text{diam}(\GG_0)}\1_{x \in \scrC_{\infty, \alpha}}$, which is in $L^p(\Q_\alpha)$ for all~$p\ge1$, the right-hand side of \eqref{eqn:MaxBoundedBySum} involves a sum over $N=N(\omega)$ random variables $M_x^r \1_{x \in \scrC_{\infty, \alpha}}$ (for $x\in\scrC_{\infty,\alpha}$ with $|x|_\infty \leq \text{diam}(\GG_0)$). These are in $L^p(\Q_\alpha)$, with a uniform bound on the norm, for all~$p\ge1$. 
By Lemma 4.5 from Berger and Biskup~\cite{Berger-Biskup}, the last sum in 
\eqref{eqn:MaxBoundedBySum} is thus also in $L^p(\Q_\alpha)$ for any $p \geq 1$. 
The Cauchy-Schwarz inequality then shows the same for $E_\omega^0(T_1 M_{\hat X_1}^r)$.
\end{proof}

Lemma~\ref{lemma-3.6}, Assumption~\ref{A5} and the pigeon-hole principle then imply:
\begin{lemma}
\label{lemma-3.7}
Fix $r \geq 1$. There exists $C=C(\omega)$ such that for any $\ell>r$ and 
$\BbbP_\alpha$-a.s. every $\omega$, 
\begin{equation}
\label{E:3.27}
E_\omega^0\biggl(\,\sum_{\ul{i} \in I(r,\ell)} \prod_{k=1}^{r} T_{i_k}\,;\,\hat X_\ell=0\biggr)
\le C(\omega) \ell^{r-d/2}.
\end{equation}
\end{lemma}

\begin{proofsect}{Proof}
Let $\ell$ and~$r$ such that $\ell>r\ge1$ and pick $\ul{i}=(i_1,\dots,i_r)\in I(r,\ell)$. Then there is $s\in\{0,\dots,r\}$ such that $i_{s+1}-i_s\ge\frac \ell{r+1}$, where we set $i_0:=0$ and 
$i_{r+1}:=\ell+1$. Hence, summing over possible~$s$ and conditioning on $\hat X_{i_s}=x$ and $\hat X_{i_{s+1}-1}=y$, we find that the l.h.s. of \eqref{E:3.27}
is bounded above by 
\begin{equation}
\label{E3.27a}
\sum_{s=0}^r \,
\sum_{\begin{subarray}{c}
\ul{i} \in I(r,\ell)\\ i_{s+1}-i_s\ge\frac \ell{r+1}
\end{subarray}}
\sum_{x,y\in\scrC_{\infty,\alpha}}
E_\omega^0 \left(\prod_{k=1}^{r} T_{i_k} \; ;\;\;\;  \hat X_{i_s}=x,\; \hat X_{i_{s+1}-1}=y ,\; \hat X_\ell=0 \right).
\end{equation}
By the Markov property and reversibility each term in the sum is bounded above by
\begin{equation}
  C(\omega) E_\omega^0 \left( \prod_{k=1}^{s} T_{i_k} \; ; \;\;\; \hat X_{i_s}=x \right)
    \,\hat{\cmss P}_\omega^{i_{s+1}-i_s-1}(x,y)\,
  E_\omega^0 \left( \prod_{k=s+1}^{r} T_{\ell-i_k+1} \; ; \;\;\; \hat X_{\ell-i_{s+1}+1}=y\right)
\end{equation}
Since $\textd_\omega(0,x) \leq \ell/2 \leq (i_{s+1}-i_s-1) (r+1)/2$ (otherwise the corresponding term is zero), 
Assumption~\ref{A5} can be used to get $\hat{\cmss P}_\omega^{i_{s+1}-i_s-1}(x,y) \leq C(\omega)\ell^{-d/2}$.
Summing over all~$x,y \in \scrC_{\infty,\alpha}$, \eqref{E3.27a} is bounded above by
\begin{equation}
\begin{aligned}
  C(\omega)\ell^{-d/2}\sum_{s=0}^r \,\sum_{\ul{i} \in I(r,\ell)}
    E_\omega^0\biggl(\,\prod_{k=1}^{s} T_{i_k}\biggr)
    E_\omega^0\biggl(\,\prod_{k=s+1}^{r} T_{\ell-i_k+1}\biggr)
  & \le C(\omega)\ell^{r-d/2}\sum_{s=0}^r M_0^s M_0^{r-s} \\
  & \le C(\omega) (r+1) \ell^{r-d/2}
\end{aligned}
\end{equation}
where $M_0^r$ is defined in \eqref{E:3.23} and the last inequality follows from Lemma~\ref{lemma-3.6}.
\end{proofsect}

\begin{proofsect}{Proof of Proposition~\ref{prop-3.5}}
As in the proof of Proposition~\ref{prop-3.4}, bound the sum on the left-hand side of \eqref{E:3.21}
by
\begin{equation}
\label{eqn:3.33}
\sum_{\ell=1}^\theta P_\omega^0\Bigl(\,\hat X_\ell=0,\,(T_1,\dots,T_\ell)\not\in G_{\,\ell,r}(\theta),\,\sum_{i=1}^\ell T_i\ge n \Bigr).
\end{equation}
We will convert the probability into expectation by invoking \eqref{E:4.15d}.
Using the Markov inequality, \eqref{eqn:3.33} is thus bounded above by 
\begin{equation}
\frac{(r+1)!}{n\theta^r}
\sum_{\ell=1}^\theta E_\omega^0\biggl(\,\sum_{\ul{i} \in I(r+1,\ell)} \prod_{k=1}^{r+1} T_{i_k}\,;
  \;\hat X_\ell=0\biggr).
\end{equation}
By Lemma~\ref{lemma-3.7} the expectation is bounded by $C(\omega)\ell^{r+1-d/2}$. Since $r+1-d/2\ge-\ffrac12$, the sum is of order~$\theta^{r-d/2+2}$. Combining this with the prefactor, the claim follows.
\end{proofsect}

\begin{proofsect}{Proof of Theorem~\ref{thm:trapping}}
Combining Propositions~\ref{prop-3.4}, \ref{prop-3.5} and setting $\Delta:=\theta$, we get,
\begin{equation}
\label{eqn:4.17}
\frac1{\theta+1}\sum_{m=n-\theta}^n P_\omega^0\bigl(\EE_{\omega,\alpha}^{(r)}(\theta,m)^\cc,\,X_{2m}=0\bigr)
\le\frac{C(\omega)}{n^{d/2}}
\begin{cases}
\bigl(\tfrac n\theta\bigr)\log(\ffrac n\theta),\qquad&\text{if }d=4,
\\*[2mm]
\bigl(\tfrac n\theta\bigr)^{\ffrac d2-1},\qquad&\text{if }d\ge5.
\end{cases}
\end{equation}
Now let $n^\star$ be the index for which the corresponding term on the left-hand side is minimal. The claim then follows by noting that
\begin{equation}
P_\omega^0\bigl(\EE_{\omega,\alpha}^{(r)}(\theta,n^\star)^\cc\big|X_{2n^\star}=0\bigr)
\le P_\omega^0\bigl(\EE_{\omega,\alpha}^{(r)}(\theta,n^\star)^\cc,\,X_{2n^\star}=0\bigr)
\frac1{\hat{\cmss P}_\omega^{2n}(0,0)}
\end{equation}
as implied by $\hat{\cmss P}_\omega^{2n}(0,0)\le \hat{\cmss P}_\omega^{2n^\star}(0,0)$ due to $n^\star\le n$.
\end{proofsect}

\section{Refinements under strongly subdiffusive decay}
\label{sec4}\noindent
The goal of this section is to prove Theorem~\ref{thm:enhanced} which eliminates the need for choosing~$n^\star$ under the assumption of a strong subdiffusive decay.  Our general strategy is as follows: Since we already know that on $\{X_{2n^\star}=0\}$ the walk~$X$ spends time of order~$n$ in one of the connected components of $\Z^d\setminus\CC_{\infty,\alpha}$, it suffices to show that we can increase the time spent in this component by~$2(n-n^\star)$ at a negligible cost of probability. We will achieve this by conditioning on the entry and exit points~$x$, resp.,~$y$ of the walk to this component and show the following regularity estimate on the probability that the walk spends a given time in the component:

\begin{proposition}
\label{lem:T_pmf_bound}
There is $c_1=c_1(\alpha,d)$ such that for any $x,y \in \scrC_{\infty, \alpha}$ and any $n>1$ and $k \geq 1$,
\begin{enumerate}
 \item $\displaystyle P_{\omega}^x(\hatX_1=y,\, T_1=n) \leq \frac{c_1}{n^2}|\GG_x \cap \GG_y|$.
 \item $\displaystyle\bigl|P_{\omega}^x(\hatX_1=y,\, T_1=n) - P_{\omega}^x(\hatX_1=y,\, T_1={n+2k})\bigr| 
    \leq c_1 \frac{k }{n^3}|\GG_x \cap \GG_y|$.
\end{enumerate}
\end{proposition}

We note that the restriction to $n>1$ ensures that the walk~$X$ actually steps out of~$\scrC_{\infty,\alpha}$.
In order to prove these bounds, we will need some preparations. If $\GG_x\cap\GG_y = \emptyset$ there is nothing to prove, hence we suppose otherwise and set 
$\GG_{xy}:= (\GG_x\cap\GG_y)\setminus\scrC_{\infty,\alpha}$. Define
\begin{equation}
\cmss Q(z,z'):=\begin{cases}
\cmss P_\omega(z,z'),\qquad&\text{if }z,z'\in\GG_{xy},
\\
0,\qquad&\text{otherwise}.
\end{cases}
\end{equation}
Then $\cmss Q$ is a substochastic kernel on $\GG_{xy}$ which is reversible with respect to $\pi_\omega$ (restricted to~$\GG_{xy}$) and self-adjoint on $\ell^2(\GG_{xy},\pi_\omega)$. Let $\langle\cdot,\cdot\rangle$ denote the inner product in this space and let $\Vert \cmss Q\Vert$ be the corresponding (operator) norm of~$\cmss Q$.

\begin{lemma}
\label{lemma-substochastic}
We have $\Vert \cmss Q\Vert<1$ and, in particular, $1-\cmss Q^2$ is positive and invertible.
\end{lemma}

\begin{proof}
Let $n:=|\GG_{xy}|$. It is easy to check that $\cmss Q^{n}$ --- regarded as an $n\times n$ matrix --- has all row sums strictly less than one. A simple computation then implies $\Vert\cmss Q^{n}\Vert<1$. Indeed, pick $h\in\ell^2(\GG_{xy},\pi_\omega)$ and compute:
\begin{equation}
\begin{aligned}
\langle h,\cmss Q^n h\rangle&=\sum_{z,z'\in\GG_{xy}}\pi_\omega(z)\cmss Q^n(z,z') h(z)
	h(z')
\\
&\le\biggl(\,\,\sum_{z,z'\GG_{xy}}\pi_\omega(z)\cmss Q^n(z,z') \bigl|h(z)\bigr|^2\biggr)^{\ffrac12}\biggl(\,\,\sum_{z,z'\GG_{xy}}\pi_\omega(z)\cmss Q^n(z,z') \bigl|h(z')\bigr|^2\biggr)^{\ffrac12}
\\
&=\sum_{z,z'\GG_{xy}}\pi_\omega(z)\cmss Q^n(z,z') \bigl|h(z)\bigr|^2
\le\Vert h\Vert^2\max\biggl\{\,\sum_{z'\in\GG_{xy}}\cmss Q^n(z,z')\colon z\in\GG_{xy}\biggr\}.
\end{aligned}
\end{equation}
Here we used Cauchy-Schwarz and applied the symmetry of $z,z'\mapsto\pi_\omega(z)\cmss Q^n(z,z')$. The maximum is $<1$ and so $\Vert\cmss Q^n\Vert<1$ as well.

As $\cmss Q$ is self-adjoint, we have $\Vert\cmss Q^{n}\Vert=\Vert\cmss Q\Vert^{n}$ and so $\Vert\cmss Q\Vert<1$ as well. The positivity and invertibility of $1-\cmss Q^2$ directly follows.
\end{proof}

Let $\delta_z\colon\GG_{xy}\to\R$ denote the element of $\ell^2(\GG_{xy},\pi_\omega)$ such that $\delta_z(z')=1$ for $z'=z$ and zero otherwise. For $u\in\{x,y\}$, define
\begin{equation}
h_u:=\sum_{z \in \GG_{xy}} \cmss P_{\omega}(z,u)\,\delta_z.
\end{equation}
The forthcoming derivations hinge on the following functional-analytic representation of the quantities in Proposition~\ref{lem:T_pmf_bound}:

\begin{lemma}
For~$x,y\in\scrC_{\infty,\alpha}$ with $\GG_{xy}\ne\emptyset$,
\begin{equation}
\label{E:4.5}
\pi_\omega(x)\,P_{\omega}^x\bigl(\hatX_1=y,\, T_1=n,\,X_1\in\GG_{xy}\bigr)=\bigl\langle h_x, \cmss Q^{n-2} h_y\bigr\rangle.
\end{equation}
\end{lemma}

\begin{proof}
Note that we obviously have
\begin{equation}
P_\omega^x\bigl(T_1=n,\,X_1=z,\,X_{n-1}=z',\,X_n=y\bigr)
=\cmss P_\omega(x,z)\cmss Q^{n-2}(z,z')\cmss P_\omega(z',y)
\end{equation}
Now multiply both sides by $\pi_\omega(x)$ and use reversibility to write $\pi_\omega(x)\cmss P_\omega(x,z)=\cmss P_\omega(z,x)\pi_\omega(z)$. Since $\pi_\omega(z)\cmss Q^{n-2}(z,z')=\langle\delta_z,\cmss Q^{n-2}\delta_{z'}\rangle$, the result follows by summing over $z,z'\in\GG_{xy}$ and invoking the (bi)linearity of the inner product.
\end{proof}

Now we are ready to prove the desired claims (1) and (2) above: 

\begin{proof}[Proof of Proposition~\ref{lem:T_pmf_bound}]
Suppose without loss of generality that $n\gg1$ let $r\in\{0,1,2,3\}$ and $m\ge0$ be such that $n-2=4m+r$. Lemma~\ref{lemma-substochastic} tells us that $1-\cmss Q^2$ is positive and invertible. The Spectral Theorem yields
\begin{equation}
\label{E:4.7}
\bigl\Vert(1-\cmss Q^2)\cmss Q^{2m}\bigr\Vert\le \frac 1{m+1},\qquad m\ge0.
\end{equation}
By Cauchy-Schwarz and the fact that $2m+r+1\ge n/2$,
\begin{equation}
\begin{aligned}
\bigl\langle h_x, \cmss Q^{n-2} h_y\bigr\rangle&
=\Bigl\langle(1-\cmss Q^2)^{-\frac12}\cmss Q^{m}h_x,(1-\cmss Q^2)\cmss Q^{2m+r}\,(1-\cmss Q^2)^{-\frac12}\cmss Q^{m}h_y\Bigr\rangle
\\
&\le\frac 2n\,
\bigl\langle h_x,(1-\cmss Q^2)^{-1}\cmss Q^{2m}h_x\bigr\rangle^{\ffrac12}
\bigl\langle h_y,(1-\cmss Q^2)^{-1}\cmss Q^{2m}h_y\bigr\rangle^{\ffrac12}.
\end{aligned}
\end{equation}
Writing $(1-\cmss Q^2)^{-1}\cmss Q^{2m}$ as a geometric series and using 
\eqref{E:4.5} we get
\begin{equation}
\begin{aligned}
\bigl\langle h_x,(1-\cmss Q^2)^{-1}\cmss Q^{2m}h_x\bigr\rangle
&=\pi_\omega(x)P_\omega^x\bigl(\hat X_1=x,\,T_1\ge 2m,\,X_1\in\GG_{xy}\bigr)
\\
&\le \pi_\omega(x)P_\omega^x(T_1\ge 2m,\,X_1\in\GG_{xy})\le \frac{\pi_\omega(x)}{2m}E_\omega^x(T_1;\,X_1\in\GG_{xy}).
\end{aligned}
\end{equation}
The argument in Lemma~3.8 of~\cite{BBHK} shows that
\begin{equation}
E_\omega^x(T_1;\,X_1\in\GG_{xy})\le c_1|\GG_{xy}|
\end{equation}
for some $c_1=c_1(d,\alpha)$, by which we conclude
\begin{equation}
P_{\omega}^x(\hatX_1=y,\, T_1=n)\le \frac {c_1}{mn}\,\sqrt{\frac{\pi_\omega(y)}{\pi_\omega(x)}}\,|\GG_{xy}|.
\end{equation}
Using that $\alpha\le\pi_\omega(z)\le2d$ for both $z=x,y$, part~(1) follows by the fact that $m\ge\frac14(n-5)$.

For the second part, it suffices to prove this for $k:=1$ with a constant that is uniform in~$n$; the general claim follows by telescoping. Instead of \eqref{E:4.7} we will need
\begin{equation}
\bigl\Vert(1-\cmss Q^2)^2\cmss Q^{2m}\bigr\Vert\le \frac 4{(m+2)^2},\qquad m\ge0.
\end{equation}
Since
\begin{equation}
\pi_\omega(x)\Bigl(P_{\omega}^x(\hatX_1=y,\, T_1=n) - P_{\omega}^x(\hatX_1=y,\, T_1={n+2})\Bigr)
=\bigl\langle h_x,(1-\cmss Q^2)\cmss Q^{n-2}h_y\bigr\rangle,
\end{equation}
once we write
\begin{equation}
\bigl\langle h_x,(1-\cmss Q^2)\cmss Q^{n-2}h_y\bigr\rangle
=\Bigl\langle(1-\cmss Q^2)^{-\ffrac12}\cmss Q^{m}h_x,(1-\cmss Q^2)^2\cmss Q^{2m+r}\,(1-\cmss Q^2)^{-\ffrac12}\cmss Q^{m}h_y\Bigr\rangle
\end{equation}
the argument proceeds exactly as above.
\end{proof}

For the remainder of this section $\alpha$ can be any value in $(0, \alpha_0)$. As before, the 
dependence of the statements (and in particular the constants) on $\alpha$ will not be 
indicated explicitly. Our next goal is to show a regularity estimate on the probability of the desired event:

\begin{proposition}
\label{prop-4.4}
Let $r \ge 1$. There is a $\BbbP_\alpha$-a.s.\ finite $C=C(\omega)$ such that, for any $1 < \Delta, \theta < \ffrac n2$ and $m\in\{n-\Delta,\dots,n\}$,
\begin{equation}
\label{E:4.14}
P_\omega^0\bigl(\EE_{\omega,\alpha}^{(r)}(\theta,m),\,X_{2m}=0\bigr)
-P_\omega^0\bigl(\EE_{\omega,\alpha}^{(r)}(\theta,n),\,X_{2n}=0\bigr)
\le C(\omega)\frac{\Delta}{n^3}
\begin{cases}
\log\theta,\,\,&\text{if }d=4,
\\
1 ,\,\,&\text{if }d\ge5.
\end{cases}
\end{equation}
\end{proposition}

\begin{proofsect}{Proof}
Throughout we will assume that $0\in\scrC_{\infty,\alpha}$ and that $X_0=0$. The proof is based on a path transformation argument. Let $\Delta_0:=n-m$. Given a path $\hat X_0, \hat X_1, \dots\hat X_\ell$ of the coarse-grained walk and the corresponding sequence of times $T_1,T_2,\dots,T_\ell$, we can record the pairs in the form
\begin{equation}
\label{eqn:t-path} 
\gamma:= \bigl((x_0,\; t_0),\, (x_1,\;t_1), \dots, (x_\ell,\;t_\ell)\bigr),
\end{equation}
where $x_j:=\hat X_j$ and $t_j:=T_j$. Then $\gamma$ is a ``path'' of the coarse-grained
random walk with $\ell=\ell(\gamma)$ steps. We shall identify events involving only $(\hatX_i,T_i)_{i\ge1}$ with sets of paths and write $\{\gamma\}$ for the event $\{((\hatX_1,\; T_1),\, \dots, (\hatX_\ell,\;T_\ell))= \gamma\}$.

For any $s \ge 1$, define the path transformation $\varphi_s$ as follows:
Given $\gamma$, let $k = k(\gamma)$ be the smallest index such that $t_k=\max_{1\le j\le\ell(\gamma)}t_j$ and set
\begin{equation}
\varphi_s(\gamma) :=\bigl((x_1,\; t_1),\, \dots, (x_k, t_k + 2s), \dots, (x_\ell,\;t_\ell)\bigr).
\end{equation}
It is easy to see that this is a one-to-one mapping. Furthermore, if $\gamma \in  
\EE_{\omega,\alpha}^{(r)}(\theta, m)\cap\{X_{2m}=0\}$ then $\varphi_{\Delta_0}(\gamma)
\in \EE_{\omega,\alpha}^{(r)}(\theta, n)\cap\{X_{2n}=0\}$. The last two statements imply
\begin{equation}
\label{E:4.23}
\text{l.h.s. of \eqref{E:4.14}}\le \sum_{\gamma} \bigl[P_\omega^0(\gamma)-P_\omega^0(\varphi_{\Delta_0}(\gamma))\bigr] \,,
\end{equation}
where the sum is over all $\gamma \in \EE_{\omega,\alpha}^{(r)}(\theta, m)\cap\{X_{2m}=0\}$. 

For such $\gamma$, the difference in the corresponding 
term is (with $\ell = \ell(\gamma)$, $k = k(\gamma)$)
\begin{equation}
\begin{aligned}
   \hat{\cmss P}_\omega^{k-1}(0,x_{k-1})
  \big[P_\omega^{x_{k-1}}(T_1=t_k,\,\hat X_1=x_k) - 
      &P_\omega^{x_{k-1}}(T_1=t_k + \Delta_0,\,\hat X_1=x_k)\big]
  \hat{\cmss P}_\omega^{\ell-k}(x_k,0) \\
  & \le c_1
  \hat{\cmss P}_\omega^{k-1}(0,x_{k-1})
  \frac{\Delta}{n^3}\bigl|\GG_{x_{k-1}}\cap\GG_{x_k}\bigr|
  \hat{\cmss P}_\omega^{\ell-k}(x_k,0)
\end{aligned}
\end{equation}
where we have used Proposition~\ref{lem:T_pmf_bound} for the middle term, 
the fact that $t_k \ge (2m-\theta)/r$ and the bounds on $\Delta$, $\theta$ and $m$.

Summing over $\ell=\ell(\gamma)$, $k=k(\gamma)$, $x:=x_{k-1}$ and $y:=x_k$ and using the reversibility
of $\hat{\cmss P}_\omega$, the sum in \eqref{E:4.23} is 
bounded above by
\begin{equation}
c_1\frac{\Delta}{n^3}
\sum_{\ell=1}^\theta\sum_{k=1}^\ell\sum_{x,y}\hat{\cmss P}_\omega^{k-1}(0,x)\hat{\cmss P}_\omega^{\ell-k}(0, y)|\GG_x\cap\GG_y|.
\end{equation}
Applying the bound in Assumption~\ref{A5} for $k-1$ or $\ell-k$, whichever is 
larger, and using the  bound $\sum_y |\GG_x \cap \GG_y| \le 2d |\GG_x|^2$, we get
\begin{equation}
\label{E:4.27}
\text{l.h.s. of \eqref{E:4.14}}\le  C(\omega)\frac{\Delta}{n^3}
\sum_{\ell=1}^\theta \frac1{\ell^{d/2}}E_\omega^0\biggl(\,\sum_{k=0}^{\ell-1} |\GG_{\hatX_k}|^2\biggr).
\end{equation}
The expectation is less than $C(\omega)\ell$ by our earlier arguments based on the Dominated Ergodic Theorem (namely, the proof of Lemma~\ref{lemma-Z's}), since $|\GG_0|$ has all moments
under~$\BbbP_\alpha$. The proof is finished by computing the sum.
\end{proofsect}

The regularity estimate from Proposition~\ref{prop-4.4} and the universal bounds in Propositions~\ref{prop-3.4}-\ref{prop-3.5} allow us to complete the proof of our last main theorem. The key issue (which was not there for Theorem~\ref{thm:trapping}) is that these bounds do not have the same structure and thus some optimization will be necessary.

\begin{proof}[Proof of Theorem~\ref{thm:enhanced}]
Combining Propositions~\ref{prop-3.4}, \ref{prop-3.5} and \ref{prop-4.4} we get
\begin{equation}
\label{E:4.29}
P_\omega^0\bigl(\EE_{\omega,\alpha}^{(r)}(\theta,n)^\cc,\,X_{2n}=0\bigr)
\le C(\omega)\begin{cases}
\displaystyle\frac\Delta{n^3}\log\theta + \frac{\log(\ffrac n\theta)}{n\Delta},\qquad&\text{if }d=4,
\\*[4mm]
\displaystyle\frac\Delta{n^3} + \frac1{n\Delta \theta^{\frac d2-2}},\quad&\text{if }d\ge5.
\end{cases}
\end{equation}
We may now minimize the above expression on $\Delta \in [1,\ffrac n2]$ and compare the result with $\cmss P_\omega^{2n}(0,0)$. The argument proceeds by considering separately three ranges of~$d$.

We first treat the case $5\le d\le 8$. Here we set
\begin{equation}
\label{eqn:OptimalDelta}
\Delta:=\bigl\lfloor \tfrac12 n\theta^{1-\ffrac d4} \bigr\rfloor \, \in [1,\ffrac n2] \,.
\end{equation}
Recalling the definition of~$\zeta_n$ from \eqref{eqn:Zeta}, we get 
\begin{equation}
\frac\Delta{n^3} + \frac1{n\Delta \theta^{\frac d2-2}} \le C n^{-2}\theta^{1-\ffrac d4}=\frac{C}{\zeta_n(\omega)}\Bigl(\frac n\theta\Bigr)^{\ffrac d4-1}\cmss P_\omega^{2n}(0,0)
\end{equation}
as desired.

If $d \geq 9$, we need to be more careful in the choice of $\Delta$, since it cannot be smaller than $1$.
Therefore we set,
\begin{equation}
\label{eqn:OptimalDelta1}
\Delta:=\bigl\lfloor \tfrac14 n\theta^{1-\ffrac d4} + 1\bigr\rfloor \, \in [1,\ffrac n2] \,.
\end{equation}
Then, 
\begin{equation}
\frac\Delta{n^3} + \frac1{n\Delta \theta^{\frac d2-2}} \le 
  \frac{C \Delta}{n^3} = 
  \frac{C(n\theta^{1-\ffrac d4} + 1)}{n^{3}}
  \le \frac{C(n\theta^{1-\ffrac d4} + 1)}{\zeta_n(\omega)}
\cmss P_\omega^{2n}(0,0) \, .
\end{equation}

Finally, it remains to establish the case $d=4$. Here we set
\begin{equation}
\label{E:4.32}
\Delta:=\lfloor \tfrac12 n\sqrt{\log (n/\theta)/\log n} + 1 \rfloor \, \in [1,\ffrac n2] \, .
\end{equation}
Then,
\begin{equation}
\label{E:4.33}
\frac\Delta{n^3}\log\theta + \frac{\log(\ffrac n\theta)}{n\Delta}
\le \frac{C \log(\ffrac n\theta)}{n\Delta}
\le \frac{C}{n^2}\sqrt{\log(\ffrac n\theta)\log n}
=\frac{C}{\zeta_n(\omega)}\sqrt{\log(\ffrac n\theta)}\,\,\cmss P_\omega^{2n}(0,0)
\end{equation}
and the claim is proved in this case as well.
\end{proof}

\section*{Acknowledgments}
\noindent
This research has  been partially supported by the NSF grant DMS-1106850, NSA grant H98230-11-1-0171 and the GA\v CR project P201-11-1558. The authors wish to express their thanks to anonymous referees for their careful reading and constructive criticism.

\end{document}